\documentclass[12pt,reqno]{amsart}

\usepackage[latin1]{inputenc}
\usepackage{amsmath}
\usepackage{amsfonts}
\usepackage{amssymb}
\usepackage{graphics}
\usepackage{enumerate}
\usepackage{subfigure}
\usepackage{amssymb,amsmath,amsthm,amscd,epsf,latexsym,verbatim,graphicx,amsfonts}
\input epsf.tex

\usepackage{amscd}
\usepackage{amsmath}
\usepackage{amssymb}
\usepackage{amsthm}
\usepackage{epsf}
\usepackage{latexsym}
\usepackage{verbatim}
\input epsf.tex

\theoremstyle{plain}
\newtheorem{theorem}{Theorem}[section]

\newtheorem{corollary}[theorem]{Corollary}

\theoremstyle{definition}

\theoremstyle{remark}

\newcommand{\nri}{n\rightarrow\infty}
\newcommand{\bnri}{N\rightarrow\infty}

\newcommand{\bbR}{\mathbb{R}}

\newcommand{\bbN}{\mathbb{N}}
\newcommand{\bbS}{\mathbb{S}}

\newcommand{\mca}{\mathcal{A}}
\newcommand{\mcn}{\mathcal{N}}

\newcommand{\mcm}{\mathcal{M}}
\newcommand{\mcp}{\mathcal{P}}

\newcommand{\diam}{\textrm{diam}}

\newcommand{\mue}{\mu_{\textrm{eq}}}

\DeclareMathOperator*{\supp}{supp}

\title[]{Extremal Polarization Configurations for Integrable Kernels}
\date{}

\begin{document}
\maketitle

\begin{center}
\textrm{Brian Simanek  \footnote{\label{note1}The author gratefully acknowledges support from Doug Hardin and Ed Saff's National Science Foundation grant DMS-1109266}}%,  Alex Vlasiuk \footnotemark[1]}
\end{center}

\begin{abstract}
Our main result shows that if a lower-semicontinuous kernel $K$ satisfies some mild additional hypotheses, then asympotitically polarization optimal configurations are precisely those that are asymptotically distributed according to the equilibrium measure for the corresponding minimum energy problem.
\end{abstract}

%\vspace{4mm}

%\footnotesize\noindent\textbf{Keywords:} Normal Derivative Behavior

%\vspace{2mm}

%\noindent\textbf{Mathematics Subject Classification:} Primary 42C05; Secondary 

%\vspace{2mm}

\normalsize

\section{Background and Results}\label{res}

%\section{Results}\label{results}

Suppose $\mca$ is a compact set that is embedded in $\bbR^t$.  Let $K(x,y):\mca\times\mca\rightarrow[0,\infty]$ be a kernel given by $K(x,y)=f(|x-y|)$, where $f:[0,\infty)\rightarrow[0,\infty]$ is a lower semi-continuous function and $|\cdot|$ represents the Euclidean distance in $\bbR^t$.  We will let $\mcm(\mca)$ denote the set of positive probability measures with support in $\mca$.  For any $\mu\in\mcm(\mca)$, the kernel generates a potential $U^{\mu}$ by
\[
U^{\mu}(x)=\int_{\mca}K(x,y)d\mu(y),\qquad x\in\bbR^t,%x\in\mca
\]
which is also non-negative and lower semi-continuous (by Fatou's Lemma; see \cite[Section 1.2]{MZ}).  For any configuration $\omega_N=(a_1,\ldots,a_N)$ of $N$ (possibly not distinct) points in $\mca$, we define its \textit{polarization} by
\[
P(\omega_N):=\min_{x\in\mca}\frac{1}{N}\sum_{y\in\omega_N}K(x,y).
\]
Equivalently, $P(\omega_N)$ is the minimum in $\mca$ of the potential generated by the probability measure $\nu_N$ that assigns weight $N^{-1}$ to each point in $\omega_N$ (counting multiplicities).  
If we associate such $N$-point configurations with the space $\mca^N$, then the extremal $N$-point  polarization problem is to find
\[
\mcp(\mca,N):=\sup_{\omega_N\in\mca^N}P(\omega_N).
\]
If $\mcm_N(\mca)$ denotes the set of all probability measures $\nu$ of the form
\[
\nu=\frac{1}{N}\sum_{j=1}^N\delta_{a_j},\qquad\qquad a_j\in\mca,\quad j=1,\ldots,N,
\]
then $\mcp(\mca,N)$ can be defined as
\[
\mcp(\mca,N)=\sup_{\nu\in\mcm_N(\mca)}\,\min_{x\in\mca}U^{\nu}(x).
\]

Polarization problems have a lengthy history, with many substantial results appearing in \cite{ABE,MEBook,ES,FR,HKS,Ohtsuka}.  One of the most fundamental results is \cite[Theorem 2]{Ohtsuka}, which asserts that
\begin{align}\label{ohtresult}
\lim_{\bnri}\mcp(\mca,N)=\sup_{\mu\in\mcm(\mca)}\min_{x\in\mca}U^{\mu}(x).
\end{align}
Any measure $\mu$ that achieves the supremum on the right-hand side of (\ref{ohtresult}) is called a \textit{polarization extremal measure}.  One consequence of our results will be a demonstration of the uniqueness of the polarization extremal measure for a large class of kernels $K$ and compact sets $\mca$ and a proof that these extremal measures are also extremal for the minimum energy problem.

The minimum energy problem for the kernel $K$ and the set $\mca$ is to find a configuration $\omega_N=(a_1,\ldots,a_N)\in\mca^N$ that minimizes then the energy functional
\[
E(\omega_N):=\sum_{{i,j=1}\atop{i\neq j}}^NK(a_i,a_j).
\]
It is well-known that if there is $\mu\in\mcm(\mca)$ so that
\[
I[\mu]:=\int_{\mca}\int_{\mca}K(x,y)d\mu(x)d\mu(y)<\infty,
\]
then there is a measure $\mue\in\mcm(\mca)$ so that
\[
\lim_{\bnri}\frac{\min_{\omega_N\in\mca^N}E(\omega_N)}{N^2}=I[\mue].
\]
The quantity $I[\mu]$ is known as the \textit{$K$-energy} of $\mu$ and the measure $\mue$ is known as a \textit{$K$-equilibrium measure}.  The set of $K$-equilibrium measures is given by
\[
\left\{\mu\in\mcm(\mca): I[\mu]=\inf_{\nu\in\mcm(\mca)}I[\nu]\right\},
\]
(see \cite{LGS}).
%A second characterization is that $\mu\in\mcm(\mca)$ is a $K$-equilibrium measure precisely when $U^{\mu}$ is constant quasi-everywher on $\mca$ (the precise definition of quasi-everywhere is not important for us, but it means that the exceptional set has $K$-capacity zero).

For our computations, we will make the following additional assumptions on the kernel $K$ and the set $\mca$:

\vspace{2mm}

(A1)  There is a $\mu\in\mcm(\mca)$ so that $I[\mu]<\infty$.

\vspace{2mm}

(A2)  The kernel $K$ has a unique equilibrium measure, which we denote by $\mue$.

\vspace{2mm}

%\noindent  There are several properties of the kernel $K$ that could imply uniqueness of the equilibrium measure.  For example, it is true if the kernel $K$ is positive definite in the sense of \cite[page 79]{Landkof}.

%\vspace{2mm}

(A3)  The support of $\mue$ is all of $\mca$.

\vspace{2mm}

%\vspace{2mm}

(A4)  The potential function
\[
U^e(x):=\int_{\mca}K(x,y)d\mue(y)
\]
is equal to a positive constant on all of $\mca$, which we denote by $W_K$.

%\vspace{2mm}

%\noindent  In general, the potential $U^e$ is constant quasi-everywhere on the support of $\mue$ (see \cite{Landkof}).  We will make the stronger assumption that it is constant everywhere on $\mca$.

\vspace{4mm}

The conditions (A1-A4), while far from being generic, are satisfied by a very large collection of compact sets $\mca$ and kernels $K$ and we will explore some examples in Section \ref{exam}.  Note that condition (A3) is not heavily restrictive in the sense that if $\supp(\mue)\neq\mca$, then we can redefine $\mca$ to be the support of $\mue$ so that (A3) is then satisfied.  All four conditions are satisfied when $\mca=\bbS^d\subset\bbR^t$ and $K(x,y)=|x-y|^{-s}$ for any $s\in(0,d)$.  In this case, the $K$-equilibrium measure is normalized surface-area measure on $\bbS^d$ (see \cite{Notices}).

Now we are ready to state our main result.

\begin{theorem}\label{equi}
Let the compact set $\mca$ and lower semi-continuous kernel $K$ satisfy conditions (A1-A4).  For each $N\geq2$, choose some $\omega_N\in\mca^N$ and let $\nu_N$ be the probability measure that assigns mass $N^{-1}$ to each point in $\omega_N$ (counting multiplicities).  The following are equivalent:
\begin{enumerate}
\item  The measures $\{\nu_N\}_{N\geq2}$ converge weakly to $\mue$.
\item  It holds that
\[
\lim_{\bnri}P(\omega_N)=W_K.
\]
\item  It holds that
\[
\lim_{\bnri}\left(\frac{1}{N}\sum_{y\in\omega_N}K(x,y)-P(\omega_N)\right)=0,
\]
in $L^1(\mue)$.
\end{enumerate}
\end{theorem}

\noindent\textit{Remark.}  In \cite{BB}, Borodachov and Bosuwan showed that if $K(x,y)=|x-y|^{-d}$ and $\mca$ is a $d$-dimensional manifold, then any sequence of polarization optimal configurations is asymptotically equidistributed on $\mca$ as $\nri$.  This is distinct from our results because the kernel does not satisfy (A1). 

\begin{proof}
Assume that (a) is true.  For every $n\geq2$, define
\[
U_n(x):=\int K(x,y)d\nu_n(y)=\frac{1}{n}\sum_{y\in\omega_n}K(x,y).
\]
It is clear (by Fubini's Theorem) that
\begin{align}\label{ming}
P(\omega_N)=\min_{x\in\mca}U_n(x)\leq\int U_n(x)d\mue(x)=W_K.
\end{align}
 Let $\{f_{\delta}\}_{\delta>0}$ be a collection of non-negative continuous functions on $[0,\diam(\mca)]$ converging pointwise to $f$ from below as $\delta\rightarrow0^+$.  Let $x_n$ be a point in $\mca$ where $U_n$ attains its minimum.  By passing to a subsequence if necessary, we may assume that $x_n$ converges to some $x_{\infty}$ and $U_n(x_n)$ converges to $\liminf U_m(x_m)$ as $\nri$.
 
 Let $\gamma>0$ be fixed.  Since $f_{\delta}$ is uniformly continuous, when $n$ is sufficiently large we have
\begin{align}
\nonumber U_n(x_n)&=\int K(x_n,y)d\nu_n(y)\geq\int f_{\delta}(|x_n-y|)d\nu_n(y)\geq\int f_{\delta}(|x_{\infty}-y|)d\nu_n(y)-\gamma\\
\label{glow}&\rightarrow\int f_{\delta}(|x_{\infty}-y|)d\mue(y)-\gamma,
\end{align}
as $\nri$.  Taking the supremum over all $\delta>0$ shows
\[
\liminf_{\nri}U_n(x_n)\geq\int K(x_{\infty},y)d\mue(y)-\gamma= W_K-\gamma,
\]
where we used assumption (A4) in this last equality.  Since $\gamma>0$ was arbitrary, this proves part (b).

Now let us assume (b) is true.  We know from (\ref{ming}) that
\begin{align}\label{rem}
\int U_n(x)d\mue(x)=W_K.
\end{align}
However, our assumption (b) implies $\min U_n(x)\rightarrow W_K$ as $\nri$.  Therefore, the functions $\{U_n\}_{n\geq2}$ are such that the minima converge to the average, which is $n$-independent.  We then calculate
\begin{align*}
\int_{\mca}\left|U_n(x)-\min_{z\in\mca}U_n(z)\right|d\mue(x)=\int_{\mca}\left(U_n(x)-\min_{z\in\mca}U_n(z)\right)d\mue(x)\rightarrow0,\qquad\nri,
\end{align*}
which proves (c).

Now, let us assume that (c) is true.  By appealing to (\ref{rem}), we can write
\[
W_K-P(\omega_N)=\int\left(U_n(x)-\min_{z\in\mca}U_n(z)\right)d\mue(x)\rightarrow0,
\]
as $\nri$, which proves (b).

Finally, assume (b) is true and let $\mcn\subseteq\bbN$ be a subsequence through which $\nu_n$ converges weakly to a limit $\nu_{\infty}$ as $\nri$ through $\mcn$.  We have already seen that (b) implies (c), so $U_n-W_K$ converges to $0$ in probability as $\nri$ through $\mcn$.  We may therefore take a further subsequence $\mcn_1\subseteq\mcn$ so that $U_n$ converges to $W_K$ $\mue$-almost everywhere as $\nri$ through $\mcn_1$ (see \cite[page 169]{Taylor}).  If we use the functions $\{f_{\delta}\}_{\delta>0}$ defined above, then we calculate for $\mue$-almost every $x$:
\begin{align*}
W_K&=\lim_{{\nri}\atop{n\in\mcn_1}}U_n(x)=\lim_{{\nri}\atop{n\in\mcn_1}}\int K(x,y)d\nu_n(y)\geq\limsup_{{\nri}\atop{n\in\mcn_1}}\int f_{\delta}(|x-y|)d\nu_n(y)\\
&=\int f_{\delta}(|x-y|)d\nu_{\infty}(y).
\end{align*}
Taking the supremum over all $\delta>0$ shows
\begin{align}\label{luck}
U^{\nu_{\infty}}(x)\leq W_K
\end{align}
$\mue$-almost everywhere, in particular at all isolated points of $\mca$ (by (A3)).  Finally, we note that the potential on the left-hand side of (\ref{luck}) is lower-semicontinuous as a function of $x$.  Therefore (\ref{luck}) holds for all $x\in\mca$ that are not isolated points of $\mca$, and hence for all $x\in\mca$.  From this, it follows that $\nu_{\infty}$ has the same $K$-energy as $\mue$, and the uniqueness of the $K$-equilibrium measure implies that $\nu_{\infty}$ must be $\mue$, which proves (a).
\end{proof}

\noindent\textit{Remark.}  Notice that the implication (b)$\Rightarrow$(c) in Theorem \ref{equi} does not make use of assumption (A3).

\begin{corollary}
Assume the hypotheses of Theorem \ref{equi} on $\mca$ and $K$.
\begin{enumerate}
\item[i)]  $\lim_{\bnri}\mcp(\mca,N)=W_K$
\item[ii)]  For any sequence $\{\omega_N^*\}_{N\geq2}$ of polarization optimal configurations having corresponding counting measure $\{\nu_N^*\}_{N\geq2}$, it holds that $\nu_N^*$ converges weakly to $\mue$ as $\bnri$
\item[iii)]  $\mue$ is the unique polarization extremal measure.
\end{enumerate}
\end{corollary}

\begin{proof}
(i)  As in (\ref{ming}), we have $\mcp(\mca,N)\leq W_K$.  Now, if $\{\omega_N\}_{N\geq2}$ is such that $\nu_N$ converges weakly to $\mue$ as $\bnri$, then we have
\[
\mcp(\mca,N)\geq P(\omega_N)\rightarrow W_K,
\]
as $\bnri$ by Theorem \ref{equi}.

(ii)  This is immediate from the equivalence of (a) and (b) in Theorem \ref{equi}.

(iii)  Let $\mu_p$ be a polarization extremal measure and $U_p(x)$ the corresponding potential.  Then by definition,
\[
\min_{x\in\mca}U_p(x)=W_K.
\]
However, $\int U_p(x)d\mue(x)=W_K$, so $U_p(x)=W_K$ $\mue$-almost everywhere.  Since (A3) implies $\supp(\mue)=\mca$ and $U_p(x)$ is lower-semicontinuous, this implies $U_p(x)\leq W_K$ on all of $\mca$ (as in the proof of Theorem \ref{equi}).  Therefore, $\mu_p$ has $K$-energy equal to $W_K$ and hence must be $\mue$ by (A2).
\end{proof}

\section{Examples}\label{exam}

In this section we will explore some examples that highlight the utility and some subtleties of the results of Section \ref{res}.  Throughout this section we will refer to the notion of asymptotic optimality, which we define as in \cite{BB}.  A sequence of configurations $\{\omega_N\}_{N=2}^{\infty}$ (where each $\omega_N\in\mca^N$) is said to be asymptotically optimal for the polarization problem if
\[
\lim_{\bnri}\frac{P(\omega_N)}{\mcp(\mca,N)}=1.
\]
%Similarly, we say that the sequence of configurations is asymptotically optimal for the energy problem if
%\[
%\lim_{\bnri}\frac{E(\omega_N)}{\min_{X\in\mca^N}E(X)}=1.
%\]
With this terminology, Theorem \ref{equi} can be restated as a collection of statements that are equivalent to the condition of asymptotic optimality of the sequence of point configurations $\{\omega_N\}_{N\geq2}$.

\subsection{Example: Riesz potentials on the solid ball.}  Let us assume $t\geq2$ and $\mca=\{x\in\bbR^t:|x|\leq1\}$ and consider the Riesz kernel $K(x,y)=|x-y|^{-s}$ for some $0<s\leq t-2$.  It was shown in \cite[Section 3]{ES} that the $N$-point configuration consisting of $N$ points at the origin is in fact optimal for the polarization problem on the solid ball with this choice of kernel.  It is obvious that a point mass has infinite energy, so the counting measures for the optimal polarization configurations do not, in this case, converge weakly to the equilibrium measure.  Thus we see that it is not clear how asymptotically optimal polarization configurations behave when the conditions (A1-A4) are not satisfied.  This example shows that the equivalences stated in Theorem \ref{equi} need not hold in general.

\subsection{Example: Random and greedy point configurations.}  Suppose that $\mca$ and $K$ are such that conditions (A1-A4) are satisfied.  For each $N\geq2$, let $\omega_N$ be a collection of $N$ points in $\mca$ chosen at random with distribution $\mue$ and let $\nu_N$ be the probability measure assigning weight $N^{-1}$ to each point in $\omega_N$.  The Strong Law of Large Numbers implies that as $\bnri$, the measures $\{\nu_N\}_{N\geq2}$ almost surely converge weakly to $\mue$.  Theorem \ref{equi} implies that $P(\omega_N)\rightarrow W_K$ as $\bnri$.    Therefore, randomly chosen points from the appropriate distribution are almost surely asymptotically optimal for the polarization problem.

In \cite{LGS}, L\'{o}pez-Garc\'{i}a and Saff studied greedy energy points, which are sequences of $N$-point configurations $\{\omega_N\}_{N\geq2}$ that are optimal for the energy problem subject to the constraint that $\omega_{N-1}\subseteq\omega_{N}$.  More precisely, we define a sequence $\{a_n\}_{n=1}^{\infty}$ by choosing $a_1\in\mca$ arbitrarily, and then for each $n>1$ we choose $a_n\in\mca$ so that
\[
\frac{1}{n-1}\sum_{i=1}^{n-1}K(a_n,a_i)=P((a_i)_{i=1}^{n-1}).
\]
The set $\omega_N$ is then taken to be $(a_i)_{i=1}^N$.  We recall \cite[Theorem 2.1(iii)]{LGS}, which says that under the assumptions (A1-A4), it holds that
\[
\lim_{\nri}\frac{1}{n-1}\sum_{i=1}^{n-1}K(a_n,a_i)=W_K.
\]
In other words, the sequence of configurations $\{\omega_N\}_{N\geq2}$ is asymptotically optimal for the polarization problem.  By Theorem \ref{equi}, we conclude that the measures
\[
\frac{1}{N}\sum_{i=1}^N\delta_{a_i}
\]
converge weakly to $\mue$, wihch is the same conclusion as \cite[Theorem 2.1(ii)]{LGS}.

%This is in distinct contrast to the energy problem.  Indeed, if $\mca$ is the unit circle in $\bbR^2$ and the function $f$ is decreasing and convex, then it is straightforward to prove that configurations that minimize the $K$-energy are precisely those that consist of $N$ equally spaced points (the same was shown for polarization by Hardin, Kendall, and Saff in \cite{HKS}).

\subsection{Example: Logarithmic potentials on curves in the plane.}  Consider the case when $\mca$ is a union of $M\geq1$ disjoint and mutually exterior Jordan curves in $\bbR^2$ and $K(x,y)=-\log(c|x-y|)$, where $c>0$ is a constant chosen to ensure that $K(x,y)>0$ when $x,y\in\mca$.  In this case, it is easily seen that condition (A1) is satisfied and \cite[Theorem I.1.3]{SaffTot} assures us that (A2) is satisfied.  By \cite[Theorem IV.1.3]{SaffTot} and an application of Mergelyan's Theorem (see \cite[Theorem 20.5]{Rudin}), one can check that condition (A3) is satisfied as well.

The only condition that remains to verify before we can apply our results is (A4).  There are several criteria that imply continuity of the logarithmic equilibrium potential.  For example, \cite[Theorem I.5.1]{SaffTot} tells us that if $z_0\in\mca$ and we define (for some $\lambda\in(0,1)$)
\[
A_n(z_0):=\left\{z:z\in\mca,\quad\lambda^{n+1}\leq|z-z_0|<\lambda^n\right\},
\]
then
\[
\sum_{n=1}^{\infty}\frac{-n}{\log(\mbox{cap}(A_n(z_0)))}=\infty
\]
implies continuity of the logarithmic equilibrium potential at $z_0$.  The criterion that we will use is \cite[Theorem I.4.8ii]{SaffTot}, which applies to every point of $\mca$ because every point of $\mca$ is on the boundary of two components of $\bbR^2\setminus\mca$, one of which is bounded and one of which is unbounded.  Applying this result shows condition (A4) is satisfied, and hence Theorem \ref{equi} applies in this setting.

%To make things more concrete, let us specify to the case when $\mca$ is a single analytic Jordan curve.  If $\psi$ is the conformal bijection mapping the unbounded component of $\bbC\setminus\mca$ to $\bbC\setminus\bard$ sending infinity to itself and satisfying
%\[
%\lim_{|z|\rightarrow\infty}\frac{\psi(z)}{z}>0,
%\]
%then the logarithmic equilibrium measure $\mue$ is given by $\mue(X)=\mcl(\psi(X))$, where $\mcl(X)$, denotes the arc-length measure of a set $X\subseteq\{z:|z|=1\}$.  We have already seen that $\mue$ has support equal to $\mca$ and hence \cite[Theorem I.4.8ii]{SaffTot} implies that condition (A4) is satisfied.  Therefore, Theorem \ref{equi} applies in this setting.

%\subsection{Separation and Covering}

%For this section, we modify an argument in \cite{KS}.

\vspace{3mm}

\noindent\textbf{Acknowledgements.}  It is a pleasure to thank Tim Michaels, Yujian Su, and Oleksandr Vlasiuk for much useful discussion about the contents of this paper.

%\newpage

%\vspace{7mm}

%\noindent \textsc{Tim Michaels, Vanderbilt University Department of Mathematics}

%\noindent \texttt{timothy.j.michaels$\at$vanderbilt.edu}

%\vspace{4mm}

%\noindent \textsc{Brian Simanek, Vanderbilt University Department of Mathematics}

%\noindent \texttt{brian.z.simanek$\at$vanderbilt.edu}

%\vspace{4mm}

%\noindent \textsc{Yujian Su, Vanderbilt University Department of Mathematics}

%\noindent \texttt{yujian.su$\at$vanderbilt.edu}

%\vspace{4mm}

%\noindent \textsc{Alex Vlasiuk, Vanderbilt University Department of Mathematics}

%\noindent \texttt{oleksandr.vlasiuk$\at$vanderbilt.edu}

\end{document}